\documentclass[reqno]{amsart}
\usepackage{amssymb,amscd,verbatim, amsthm,graphicx, color}
\usepackage{fancybox}
\usepackage{longtable}

\renewcommand{\Im}{\mathrm{im}}

\newcommand{\Irr}{\mathsf{Irr}}
\newcommand{\mr}{\mathsf{r}}
\newcommand{\msr}{\mathsf{sr}}

\newcommand{\even}{\mathsf{even}}
\newcommand{\odd}{\mathsf{odd}}
\newcommand{\ol}{\overline}

\newcommand \C[1]{{\mathcal #1}}

\newcommand \wti[1]{{\widetilde {#1}}}

\newcommand\fg{\mathfrak g}
\newcommand\frg{\mathfrak g}
\newcommand\frh{\mathfrak h}

\newcommand \bC{{\mathbb C}}

\newcommand \bH{{\mathbb H}}
\newcommand \bR{{\mathbb R}}

\newcommand\CB{{\C B}}
\newcommand\caD{{\C D}}

\newcommand\CN{{\C N}}
\newcommand\caN{{\CN}}

\newcommand\CR{{\C R}}
\newcommand\caR{{\CR}}

\newcommand\ep{{\epsilon}}

\newcommand\om{{\omega}}
\newcommand\al{{\alpha}}

\newcommand\fh{{\mathfrak h}}

\newtheorem{theorem}{Theorem}[section]

\newtheorem{corollary}[theorem]{Corollary}
\newtheorem{lemma}[theorem]{Lemma}
\newtheorem{proposition}[theorem]{Proposition}

\theoremstyle{definition}
\newtheorem{definition}[theorem]{Definition}
\newtheorem{remark}[theorem]{Remark}

\numberwithin{equation}{section}

\newcommand\Hom{\operatorname{Hom}}

\newcommand\tr{\operatorname{tr}}

\newcommand\triv{\mathsf{triv}}
\newcommand\sgn{\mathsf{sgn}}
\newcommand\refl{\mathsf{refl}}

\newcommand\sol{\mathsf{sol}}
\newcommand\gen{\mathsf{gen}}
\newcommand\Pin{\mathsf{Pin}}
\newcommand{\bs}{\backslash}

\begin{document}


\title[Dirac cohomology for graded affine Hecke algebras]
{Dirac cohomology  for graded affine Hecke algebras}

\author{Dan Barbasch}
       \address[D. Barbasch]{Dept. of Mathematics\\
               Cornell University\\Ithaca, NY 14850}
       \email{barbasch@math.cornell.edu}

\author{Dan Ciubotaru}
        \address[D. Ciubotaru]{Dept. of Mathematics\\ University of
          Utah\\ Salt Lake City, UT 84112}
        \email{ciubo@math.utah.edu}

\author{Peter E.~Trapa}
        \address[P. Trapa]{Dept. of Mathematics\\ University of
          Utah\\ Salt Lake City, UT 84112}
        \email{ptrapa@math.utah.edu}

\begin{abstract}
We define analogues of the Casimir and Dirac operators for graded
affine Hecke algebras, and establish a version of Parthasarathy's
Dirac operator inequality.  We then prove a version of
Vogan's Conjecture for Dirac cohomology.  The formulation of the
conjecture depends on a uniform geometric parametrization of spin
representations of Weyl groups.  Finally, we apply
    our results to the study of 
unitary representations.
\end{abstract}

\thanks{The authors were partially supported by NSF grants
DMS-0554278 and DMS-0554118.}
\maketitle

\setcounter{tocdepth}{1}

\section{Introduction}\label{sec:0}
\label{s:intro}
This paper develops the theory of Dirac cohomology
for modules over a graded affine Hecke algebra.
The cohomology of such a module $X$
is a representation
of a spin double cover $\wti W$ of a relevant Weyl group.
Our main result shows that when $X$ is irreducible, the
$\wti W$ representation (when nonzero) 
determines
the central character of $X$.  This can be interpreted as
a $p$-adic analogue of Vogan's Conjecture (proved by 
Huang and Pand\v zi\' c \cite{HP}) for Harish-Chandra modules.

In more detail, fix a root system $R$ (not necessarily 
crystallographic), let $V$ denote its complex
span, write $V^\vee$ for the complex span of the coroots, $W$
for the Weyl group,
and fix a $W$-invariant inner product $\langle~,~\rangle$ on $V^\vee$.
Let $\bH$ denote the associated graded affine Hecke algebra with parameters
defined by Lusztig \cite{L} (Definition \ref{d:graded}).   As a complex
vector space, $\bH \simeq \bC[W] \otimes S(V^\vee)$.  Lusztig proved
that maximal ideals in the center of $\bH$, and hence central characters
of irreducible $\bH$ modules, are parametrized by ($W$ orbits of) elements
of $V$.  
In particular it makes sense to speak of the length of the central character
of an irreducible $\bH$ module.  

After introducing certain Casimir-type elements in Section \ref{s:cas},
we then turn to the Dirac operator in Section \ref{sec:dirac}.   
Let $C(V^\vee)$ denote the corresponding Clifford algebra for the
inner product $\langle~,~\rangle$.  For a fixed
orthonormal basis $\{\omega_i\}$ of $V^\vee$, the Dirac operator
is defined (Definition \ref{d:dirac}) as
\[
\caD = \sum_i \wti \omega_i \otimes \omega_i \in \bH \otimes C(V^\vee)
\]
where $\wti \omega_i \in \bH$ is given by \eqref{omtilde}.
In Theorem \ref{d:dirac}, we prove $\caD$ is roughly the square root of the 
Casimir element $\sum_i\omega_i^2 \in \bH$ (Definition \ref{casimir}).

For a fixed space of spinors $S$ for $C(V^\vee)$ and a fixed $\bH$
module $X$, $\caD$ acts as an operator $D$ on $X \otimes S$.  Since
$W$ acts by orthogonal transformation on $V^\vee$, we can consider its
preimage $\wti W$ in $\Pin(V^\vee)$.  By restriction $X$ is a
representation of $W$, and so $X \otimes S$ is a representation of
$\wti W$.  Lemma \ref{l:winvdirac} shows that $\caD$ (and hence $D$)
are approximately $\wti W$ invariant.  Thus $\ker(D)$ is also a
representation of $\wti W$.  Corollary \ref{c:diracineq} shows that if
$X$ is irreducible, unitary, and $\ker(D)$ is nonzero, then any
irreducible representation of $\wti W$ occurring in $\ker(D)$
determines the length of the central character of $X$.  This is an
analogue of Parthasarathy's Dirac operator inequality \cite{Pa}
(cf.~\cite[Section 7]{SV}) for Harish-Chandra modules.

We then define the Dirac cohomology of $X$ as $H^D(X)= \ker(D)/(\ker(D) \cap \Im(D))$
in Definition \ref{d:dcoh}.  (For unitary representations $H^D(X) = \ker(D)$.)
Once again $H^D(X)$ is a representation of $\wti W$. 
At least when $\bH$ is a Hecke algebra related to $p$-adic group 
representations, one is naturally led to the 
following version of Vogan's Conjecture:
if $X$ is irreducible and $H^D(X)$ is nonzero,
then any irreducible
representation of $\wti W$ occurring in $H^D(X)$
determines the central character of $X$, not just its length.  
Our main result, Theorem \ref{t:vogan}, establishes
this for algebras $\bH$ attached to crystallographic root
systems and equal parameters.  The proof is completed
in Section \ref{s:proof}.  As explained in Remark \ref{r:geom},
the proof also applies to establish 
Theorem \ref{t:vogan} for the special kinds of unequal parameters for
which
Lusztig's geometric theory applies \cite{lu:1}-\cite{lu:3}.

To make Theorem \ref{t:vogan} precise, we
need a way of passing from an irreducible $\wti W$ representation to
a central character, i.e.~an element of $V$.  
This is a fascinating problem in its own
right.  The irreducible representations of $\wti W$ --- the so-called
spin representations of $W$ --- have been known for a long time from
the work of Schur, Morris, Reade, and others.  But only recently has a
uniform parametrization of them in terms of nilpotent
orbits emerged \cite{ciubo:weyl}.  This parametrization (partly recalled 
in Theorem \ref{t:class}) provides exactly
what is needed for the statement of Theorem \ref{t:vogan}.

One of the main reasons for introducing the Dirac operator (as in the
real case) is to study unitary representations.  We give applications
in Section \ref{s:unit}.  Corollary \ref{c:unit} and Remark
\ref{r:unit} in particular contain powerful general statements about
unitary representations.  
Given the machinery of the Dirac operator, their proofs are remarkably
simple.

We remark that while this paper is inspired by the ideas of 
Parthasarathy, Vogan, and Huang-Pand\v{z}i\'c, it is
essentially self-contained.    There are two exceptions.
We have already mentioned that we use the main
results of \cite{ciubo:weyl}.  The other
nontrivial result we need is the classification (and $W$-module
structure) of certain tempered $\bH$-modules (\cite{KL,L,lu:3}).
These results (in the form we use them)
are not available at arbitrary parameters.  This explains
the crystallographic
condition and restrictions
on parameters in the statement of Theorem \ref{t:vogan} and in
Remark \ref{r:geom}.  For applications to unitary representations
of $p$-adic groups, these hypotheses are natural.
Nonetheless we expect a version of Theorem \ref{t:vogan} to hold
 for arbitrary parameters and noncrystallographic roots systems.

The results of this paper suggest generalizations to other types of related
Hecke algebras.
They also suggest
possible generalizations
along the lines of \cite{ko} for a version of Kostant's cubic Dirac operator.
Finally, in \cite{EFM} and
\cite{ct2} (and also in
unpublished work of Hiroshi Oda)
functors between Harish-Chandra modules and 
modules for associated graded affine
Hecke algebras are introduced.  It would be interesting to
understand how these functors relate Dirac cohomology in
the two categories.

\section{Casimir operators}\label{sec:1}
\label{s:cas}

\subsection{Root systems} 
\label{ss:rs}
Fix 
a root system $\Phi=(V_0,R,V_0^\vee, R^\vee)$ over the real numbers. 
In particular: $R \subset V_0 \setminus \{0\}$ spans 
the real vector space $V_0$; 
$R^\vee \subset V^\vee \setminus \{0\}$ spans the real vector space
$V_0^\vee$; there is a perfect bilinear 
pairing 
\[
(\cdot,\cdot): V_0\times V_0^\vee\to \bR;
\] 
and there is a bijection between $R$ and $R^\vee$ denoted
$\alpha \mapsto \alpha^\vee$ such that
$(\alpha, \alpha^\vee) = 2$ for all $\alpha$.
Moreover, for $\alpha \in R$, the reflections
\begin{align*}
s_\al&: V_0\to V_0, \; \; \; s_\al(v)=v-(v,\al^\vee)\al, \\
 s_\al^\vee&:V_0^\vee\to V_0^\vee,  \; s^\vee_\al(v')=v'-(\al,v')\al^\vee
\end{align*}
leave $R$ and $R^\vee$ invariant, respectively. 
Let $W$ be the subgroup of $GL(V_0)$ 
generated by $\{s_\al \; | \; \al\in R\}$.
The map $s_\al \mapsto s^\vee_\al$ given an
embedding of $W$ into
$GL(V_0^\vee)$ so that
\begin{equation}
\label{e:w}
(v,wv') = (wv,v')
\end{equation}
for all $v \in V_0$ and $v' \in V_0^\vee$.

We will assume that the root system $\Phi$ is reduced, meaning that
$\al\in R$ implies $2\al\notin R.$ However, initially
we do not need to assume
that $\Phi$ is crystallographic, meaning that for us
$(\al,\beta^\vee)$ need not always be an integer.  We will fix a
choice of positive roots $R^+ \subset R$, let $\Pi$ denote the
corresponding simple roots in $R^+$, and let $R^{\vee,+}$ denote the
corresponding positive coroots in $R^\vee$.  Often we will write
$\al>0$ or $\al<0$ in place of $\al\in R^+$ or $\al\in (-R^+)$,
respectively.

\smallskip

We fix, as we may, a $W$-invariant inner product $\langle \cdot, \cdot \rangle$ on
$V_0^\vee$. The constructions in this paper of the Casimir and Dirac
operators depend, up to a positive scalar, on the choice of this inner
product.
Using the bilinear pairing $(\cdot, \cdot)$, we
define a dual inner product on $V_0$ as follows. Let $\{\om_i \; | \; 
i=1,\cdots,n\}$ and $\{\om^i \; | \;  i=1, \dots, n\}$ be $\bR$-bases of $V_0^\vee$ 
which are in duality; i.e.~such 
that $\langle \om_i, \om^j \rangle = \delta_{i,j}$, the Kronecker delta.
Then for $v_1, v_2 \in V_0$,  set
\begin{equation}\label{innprod}
\langle v_1, v_2 \rangle=\sum_{i=1}^n(v_1,\om_i)(v_2,\om^i).
\end{equation}
(Since the inner product on $V_0^\vee$ is also denoted $\langle \cdot , \cdot \rangle$,
this is an abuse of notation.  But it causes no confusion in practice.)
Then \eqref{innprod} defines
an inner product on $V_0$ which once again is $W$-invariant.
It does not depend on the choice of bases $\{\om_i\}$ and
$\{\om^i\}$. If $v$ is a vector in $V$ or in $V^\vee$, we set
$|v|:=\langle v,v\rangle^{1/2}.$

\subsection{The graded affine Hecke algebra}

Fix a root system $\Phi$ as in the previous section.  Set
$V=V_0\otimes_\bR \bC$, and $V^\vee=V_0^\vee\otimes_\bR \bC.$
Fix  a $W$-invariant ``parameter function'' $c:R\to \mathbb R$,
and set $c_\alpha = c(\alpha)$.

\begin{definition}[\cite{L} \S4]\label{d:graded} 
The graded affine Hecke algebra $\bH=\bH(\Phi,c)$ attached to
  the root system $\Phi$ and with parameter function $c$ is the complex associative
  algebra with unit generated by the symbols $\{t_w\; | \; w\in W\}$ 
and $\{t_f \; | \; f \in S(V^\vee)$\}, subject to the
relations:
\begin{enumerate}
\item[(1)]
The linear map from the group algebra $\bC[W] = \bigoplus_{w \in W} \bC w$ 
to $\bH$ taking $w$ to $t_w$ is an injective map of algebras.

\item[(2)]
The linear map from the symmetric algebra $S(V^\vee)$ to $\bH$ taking
an element $f$ to $t_f$ is an injective map of algebras.
\end{enumerate}
We will often implicitly invoke these inclusions and view $\bC[W]$
and $S(V^\vee)$ as subalgebras of $\bH$.
As is customary,
we also write $f$ instead of $t_f$ in $\bH$.
The final relation is
\begin{enumerate}
\item[(3)]
\begin{equation}\label{hecke}
\omega t_{s_\alpha}-t_{s_\alpha} s_\alpha(\omega)= c_\alpha
(\alpha,\omega),\quad \alpha\in \Pi,~ \omega\in V^\vee;
\end{equation}
here $s_\alpha(\omega)$ is the element of $V^\vee$ obtained by $s_\alpha$ acting on $\omega$.
\end{enumerate}

\end{definition}

Proposition 4.5 in \cite{L} says that the center $Z(\bH)$ of $\bH$ is $S(V^\vee)^W$.
Therefore maximal ideals in $Z(\bH)$ are parametrized by $W$ orbits in $V$.

\begin{definition}
\label{d:cc}
For $\nu \in V$, we write $\chi_\nu$ for the homomorphism from $Z(\bH)$ to $\bC$
whose kernel is the maximal ideal parametrized by the $W$ orbit of $\nu$.
By a version of
Schur's
lemma, $Z(\bH)$ acts in any irreducible $\bH$ module $X$ by a scalar $\chi: Z(\bH)\to
\bC$. We call $\chi$ the central character of
$(\pi,X)$.  In particular, there exists
$\nu \in V$ such that $\chi = \chi_\nu$.
\end{definition}

\subsection{The Casimir element of $\bH$}

\begin{definition}\label{d:casimir} Let $\{\om_i:i=1,n\}$ and $\{\om^i: i=1,n\}$ be dual bases of $V^\vee_0$ with respect to $\langle~,~\rangle$. 
Define 
\begin{equation}
\label{casimir}
\Omega=\sum_{i=1}^n\omega_i\omega^i\in \bH.
\end{equation}
It follows from a simple calculation that $\Omega$ is well-defined
independent of the choice of bases.
\end{definition}

\begin{lemma}
\label{l:cascentral}
The element $\Omega$ is central in $\bH$.
\end{lemma}

\begin{proof} 
To see that $\Omega$ is central, in light of Definition
\ref{d:graded}, it is sufficient to check that
$t_{s_\alpha}\Omega=\Omega t_{s_\alpha}$ for every $\alpha\in\Pi.$
Using (\ref{hecke}) twice and the fact that $(\alpha, s_\alpha(\omega)) = -(\alpha, \omega)$
(as follows from \eqref{e:w}), we find
\begin{equation}
t_{s_\alpha}(\omega_i\omega^i)=(s_\alpha(\omega_i)s_\al(\om^i)) t_{s_\alpha}+c_\al(\al,\omega^i)s_\al(\omega_i)+c_\al(\al,\omega_i)\omega^i.
\end{equation}
Therefore, we have
\begin{equation}
\begin{aligned}
t_{s_\alpha}\Omega&=\sum_{i=1}^ns_\alpha(\omega_i) s_\al(\omega^i)
t_{s_\alpha}+c_\al\sum_{i=1}^n
(\al,\omega^i)s_\al(\omega_i)+c_\al\sum_{i=1}^n(\al,\omega_i) \omega^i\\
&=\Omega
t_{s_\alpha}+c_\al\sum_{i=1}^n
(\al,s_\alpha(\omega^i))\omega_i+c_\al\sum_{i=1}^n(\al,\omega_i) \omega^i,  \\
&=\Omega t_{s_\alpha}-c_\al\sum_{i=1}^n
(\al,\omega^i)\omega_i
+c_\al\sum_{i=1}^n(\al,\omega_i) \omega^i.
\end{aligned}
\end{equation}
But the last two terms cancel (which can be seen by taking $\{\omega_i\}$ to be a
self-dual basis, for example).  So indeed
$t_{s_\alpha}\Omega
=\Omega t_{s_\alpha}$.
\end{proof}

\begin{lemma}\label{scalar} Let $(\pi,X)$ is an irreducible
  $\bH$-module with central character $\chi_\nu$ for $\nu \in V$
  (as in Definition \ref{d:cc}).
Then
$$\pi(\Omega)=\langle \nu,\nu\rangle \;  \mathrm{Id}_X.$$
\end{lemma}

\begin{proof}
Since $\Omega$ is central (by Lemma \ref{l:cascentral}), it acts
by a multiple of the identity on $X$.
We use the weight decomposition of $(\pi,X)$ with
respect to the abelian subalgebra $S(V^\vee).$ Let $x\neq 0$ be
an eigenvector for the weight $w\nu\in V$, $w\in W$. Then we have:
\begin{equation*}
\pi(\om_i\om^i)x= (w \nu,\om_i)(w\nu,\om^i) x,
\end{equation*}
and when we sum over the dual bases $\{\om_i\},\{\om^i\}$, we find
\begin{equation}
\pi(\Omega)x=\sum_{i=1}^n (w \nu,\om_i)(w\nu,\om^i) x=\langle
w\nu,w\nu\rangle x=\langle\nu,\nu\rangle x,
\end{equation} 
by (\ref{innprod}) and the $W$-invariance of $\langle~,~\rangle.$
\end{proof}

\subsection{} We will need the following formula. To simplify 
notation, we define
\begin{equation}
t_{w\beta}:=t_wt_{s_\beta}t_{w^{-1}},\text{ for $w\in W,$
$\beta\in R$}.
\end{equation}

\begin{lemma}  \label{l:formula}
For $w \in W$ and $\omega \in V^\vee$, 
\begin{equation}
  \label{eq:omt}
  t_w\om t_w^{-1}=w(\om)
  +\sum_{\beta>0\text{ s.t. }w\beta<0}c_\beta(\beta,\om) t_{w\beta}.
\end{equation}
\end{lemma}
\begin{proof}
The formula holds if $w=s_\al$, for $\al\in\Pi$, by \eqref{hecke}: 
\begin{equation}
  \label{eq:fsimple}
  t_{s_\al}\om t_{s_\al}=s_\al(\om)+c_\al(\al,\om)t_{s_\al}.
\end{equation}
We now do an induction on the length of $w.$ Suppose the formula holds
for $w,$ and let $\al$ be a simple root such that ${s_\al}w$ has strictly greater
length. Then
\begin{equation}
  \label{eq:finduction}
  \begin{aligned}
t_{s_\al}t_w\om t_{w^{-1}}t_{s_\al}&=
t_{s_\al}\left[ w(\om) +\sum_{\beta\, : \, w\beta<0}c_\beta(\beta,\om)
t_{w\beta}\right] t_{s_\al}=\\
&=s_\al w(\om)+ c_\al (\al,\om) t_{s_\al}+
\sum_{\beta \, : \, w\beta<0}c_\beta(\beta,\om) t_{s_\al}t_{w\beta}t_{s_\al}\\
&=s_\al w(\om)+c_\al(\al,\om) t_{s_\al}+
\sum_{\beta \, : \, w\beta<0}c_\beta(\beta,\om) t_{s_\al w\beta}.
\end{aligned}
\end{equation}
The claim follows.
\end{proof}

\subsection{The $*$-operation, Hermitian and unitary representations} 
The algebra $\bH$ has a natural conjugate linear
anti-involution defined on generators as follows (\cite[Section 5]{BM}):
\begin{equation} \label{eq:tomdef}
\begin{aligned}
&t_w^*=t_{w^{-1}},\quad w\in W,\\
&\omega^*=-\omega+\sum_{\beta>0}c_\beta (\beta,\omega) t_{s_\beta},\quad
\omega\in V^\vee.
\end{aligned}
\end{equation}
In general there are other conjugate linear anti-involutions on
$\bH$, but this one is distinguished by its relation to the 
canonical notion of unitarity for $p$-adic group representations
\cite{BM1}-\cite{BM}.

An $\bH$-module $(\pi,X)$ 
is said to be $*$-Hermitian (or just Hermitian) 
if there exists a Hermitian form $(~,~)_X$ on $X$ which is
invariant in the sense that:
\begin{equation}\label{hermitian}
(\pi(h)x,y)_X=(x,\pi(h^*)y)_X,\quad\text{for all }h\in\bH,~ x,y\in X.
\end{equation}
If such a form exists which is also positive definite, then $X$ is
said to be $*$-unitary (or just unitary).

Because the second formula in \eqref{eq:tomdef} is complicated, we
need other elements which behave more simply under $*$.
For every $\omega\in V^\vee$, define
\begin{equation}\label{omtilde}
\wti\om=\om-\frac 12 \sum_{\beta>0}c_\beta (\beta,\omega) t_{s_\beta}
\; \in \; \bH.
\end{equation}
Then it follows directly from the definitions that $\omega^* = - \omega$.
Thus if $(\pi,X)$ is Hermitian $\bH$-module
\begin{equation}
(\pi(\wti\om)x,\pi(\wti\om)x)_X=(\pi({\wti\om}^*)\pi(\wti\om)x,x)_X=-(\pi(\wti\om^2)x,x)_X.
\end{equation}
If we further assume that $X$ is unitary, then
\begin{equation}
\label{eq:cas0}
(\pi(\wti\om^2)x,x)_X\le 0, \quad\text{for all }x\in X,~
  \om\in V^\vee_0.
\end{equation}
For each $\omega$ and $x$, this is a necessary condition for a Hermitian
representation $X$ to be unitary.
It is difficult to apply
because the operators $\pi(\wti \omega^2)$ are intractable in general.
Instead we introduce a variation on the Casimir element of Definition 
\ref{d:casimir} whose action in an $\bH$-module will be seen to be tractable.

\begin{definition}\label{d:castilde} Let $\{\om_i\},\{\om^i\}$ be dual bases of $V^\vee_0$ with respect to $\langle~,~\rangle$. Define 
\begin{equation}
\wti\Omega=\sum_{i=1}^n \wti\omega_i\wti{\om}^i \; \in \; \bH.
\end{equation}
It will follow
from Theorem \ref{t:Omtilde} below that $\wti\Omega$ is independent
of the bases chosen.
\end{definition}

If we sum \eqref{eq:cas0} over a self-dual orthonormal basis
of $V^\vee$, we immediately obtain the following necessary condition for
unitarity.

\begin{proposition}\label{p:criterion} 
A Hermitian $\bH$-module $(\pi,X)$ 
with invariant form $(~,~)_X$ is unitary only if
\begin{equation}\label{ineq1}
(\pi(\wti\Omega)x,x)_X\le 0,\quad \text{for all }x\in X.
\end{equation}
\end{proposition}

The remainder of this section will be aimed at computing the action
of $\wti \Omega$ in an irreducible $\bH$ module as explicitly as possible
(so that the necessary condition of \ref{p:criterion} becomes as
effective as possible).  Since $\wti \Omega$ is no longer
central, nothing as simple as Lemma \ref{scalar} is available.
But Proposition \ref{p:tilde}(2) below immediately implies that $\wti \Omega$ 
invariant under conjugation by $t_w$ for $w \in W$.  It therefore acts
on each $W$ isotypic component of $\bH$ module,  and on each isotypic component
it turns out to act in a relatively 
simple manner (Corollary \ref{c:criterion}).

To get started, set
\begin{equation}\label{eq:tom}
T_\om= \om - \wti \om = \frac
12\sum_{\beta>0}c_\beta(\beta,\om)t_{s_\beta}\in\bH
\end{equation}
and
\begin{equation}\label{eq:omegaW}
\Omega_W= \frac
14\sum_{\substack{\al>0,\beta>0\\s.t.~s_\al(\beta)<0}}c_\al c_\beta \langle\al,\beta\rangle t_{s_\al}t_{s_\beta} \in \bC[W].
\end{equation}
Note that $\Omega_W$ is invariant under the conjugation action of $W$.

\begin{lemma}\label{l:tom} If $\om_1,\om_2\in V^\vee$, we have
$$[T_{\om_1},T_{\om_2}]=\frac 14\sum_{\substack{\al>0,\beta>0 \\
s.t.~s_\al(\beta)<0}} c_\al c_\beta ((\al,\om_1)(\beta,\om_2)-(\beta,\om_1)(\al,\om_2)) t_{s_\al}t_{s_\beta}.$$
\end{lemma}

\begin{proof} From the definition (\ref{eq:tom}), we see that
$$[T_{\om_1},T_{\om_2}]=\frac 14\sum_{\al>0,\beta>0} c_\al c_\beta
  ((\al,\om_1)(\beta,\om_2)-(\beta,\om_1)(\al,\om_2))
  t_{s_\al}t_{s_\beta}.$$ Assume $\al>0,\beta>0$ are such that
  $s_\al(\beta)>0.$ Notice that if $\gamma = s_\alpha(\beta)$, then
  $t_{s_\gamma}t_{s_\al}=t_{s_\al}t_{s_\beta}$. Also, it is
  elementary to verify (by a rank 2 reduction to the span of $\alpha^\vee$
  and $\beta^\vee$, for instance) that
$$(s_\al(\beta),\om_1)(\al,\om_2)-(\al,\om_1)(s_\al(\beta),\om_2)=-((\al,\om_1)(\beta,\om_2)-(\beta,\om_1)(\al,\om_2)).$$
  Since $c$ is $W$-invariant, this implies that the contributions of
  the pairs of roots $\{\al,\beta\}$ and $\{s_\al(\beta),\al\}$ (when
  $s_\al(\beta)>0$) cancel out in the above sum. The claim follows.
\end{proof}

\begin{proposition}\label{p:tilde}
Fix $w \in W$ and $\omega, \omega_1,
\omega_2 \in V^\vee$.
The elements defined in (\ref{omtilde}) have
  the following properties:

\begin{enumerate}
\item $\wti\om^*=-\wti\om$;
\item $t_w\wti\om t_{w^{-1}}=\wti{w(\om)}$;
\item $[\wti\om_1,\wti\om_2]=-[T_{\om_1},T_{\om_2}]$.
\end{enumerate}
\end{proposition}
\begin{proof}
As remarked above, property (1) is obvious from (\ref{eq:tomdef}). For (2), using
Lemma \ref{l:formula}, we have:
\begin{equation}
\begin{aligned}
t_{w}\wti\om t_{w^{-1}}&=t_w\omega t_{w^{-1}}-\frac 12
\sum_{\beta>0}c_\beta(\beta,\omega)t_wt_{s_\beta}t_{w^{-1}}\\
&=w(\omega)+\sum_{\beta>0:w\beta<0} c_\beta(\beta,\omega)
t_{w\beta}-\frac 12 \sum_{\beta>0}c_\beta(\beta,\omega)t_{w\beta}\\
&=w(\omega)+\frac
12\sum_{\beta>0:w\beta<0}c_\beta(\beta,\omega)t_{w\beta}-\frac 12 \sum_{\beta>0:w\beta>0}c_\beta(\beta,\omega)t_{w\beta}\\
&=w(\omega)-\frac 12\sum_{\beta'>0} c_{\beta'} (w^{-1}\beta', \omega)
t_{s_{\beta'}}=\wti {w(\omega)}.
\end{aligned}
\end{equation}
For the last step,
we set $\beta'=-w\beta$ in the first sum and $\beta'=w\beta$ in the second
sum, and also used that $c_{\beta'}=c_\beta$ since $c$ is $W$-invariant.

Finally, we verify (3). We have
\begin{align*}
[\wti\om_1,\wti\om_2]&=[\om_1-T_{\om_1},\om_2-T_{\om_2}]\\
&=[T_{\om_1},T_{\om_2}]-([T_{\om_1},\om_2]+[\om_1,T_{\om_2}]).
\end{align*}
 We do a direct calculation:
\[
[T_{\om_1},\om_2]=
\frac 12\sum_{\al>0}
c_\al(\al,\om_1)(t_{s_\al}\om_2 t_{s_\al}-\om_2)t_{s_\al}.
\]
Applying Lemma \ref{l:formula}, we get
\begin{align*}
[T_{\om_1},\om_2]&=\frac 12
\sum_{\al>0}c_\al (\al,\om_1)(s_\al(\om_2)-\om_2)t_{s_\alpha}
+\frac 12 \sum_{\substack{\alpha>0,\beta>0\\ s.t.~s_\al(\beta)<0}} c_\alpha c_\beta (\alpha, \om_1)(\beta,\om_2) t_{s_\al(\beta)}
t_{s_\al}\\
&=-\frac
12\sum_{\al>0}c_\al(\al,\om_1)(\al,\om_2)\al^\vee t_{s_\al} +\frac
12\sum_{\substack{\al>0,\beta>0\\s.t.~s_\al(\beta)<0}} c_\al c_\beta
(\al,\om_1)(\beta,\om_2)t_{s_\al}t_{s_\beta}.
\end{align*}
From this and Lemma \ref{l:tom}, it follows immediately that $[T_{\om_1},\om_2]+[\om_1,T_{\om_2}]=2 [T_{\om_1},T_{\om_2}]$.  This completes the proof of (3).
\end{proof}

\begin{theorem}\label{t:Omtilde} 
Let $\wti\Omega$ be the $W$-invariant element of
  $\mathbb H$ from Definition \ref{d:castilde}.  Recall the notation
of \eqref{eq:tom}.  Then 
\begin{equation}
\begin{aligned}
\wti\Omega=\Omega-\sum_{i=1}^n T_{\om_i}T_{\om^i}
&=\Omega
-
\frac
14\sum_{\substack{\al>0,\beta>0\\s.t.~s_\al(\beta)<0}}c_\al c_\beta \langle\al,\beta\rangle t_{s_\al}t_{s_\beta}.
\end{aligned}
\end{equation}
\end{theorem}
 
\begin{proof}
From Definition \ref{d:castilde}, we have
\begin{equation}\label{eq:tom2}
\wti\Omega=
\sum_{i=1}^n\om_i\om^i-\sum_{i=1}^n(\om_iT_{\om^i}+T_{\om_i}\om^i)+\sum_{i=1}^n T_{\om_i}T_{\om^i}.
\end{equation}
On the other hand, 
we have $\wti\om=(\om-\om^*)/2,$ and so
\begin{equation}\label{eq:star}
\wti\om_i\wti\om^i=(\om_i\om^i+\om_i^*{\om^i}^*)/4-(\om_i{\om^i}^*+\om_i^*\om^i)/4.
\end{equation}
Summing \eqref{eq:star} over $i$ from $1$ to $n$, we find:
\begin{equation}
  \label{eq:star2}
\begin{aligned}
\wti \Omega &=\sum_{i=1}^n\frac{\om_i\om^i+\om_i^*{\om^i}^*}{4}-\sum_{i=1}^n\frac{\om_i{\om^i}^*+\om_i^*\om^i }{4}\\&=
\frac{1}{2}\sum_{i=1}^n\om_i\om^i -\frac14\sum_{i=1}^n
[\om_i(-\om^i+2T_{\om^i})+(-\om_i+2T_{\om_i})\om^i]\\
&=
\sum_{i=1}^n\om_i\om^i-\frac12\sum_{i=1}^n(\om_iT_{\om^i}+T_{\om_i}\om^i).
\end{aligned}
\end{equation}
We conclude from \eqref{eq:tom2} and \eqref{eq:star2} that
\begin{equation}
  \label{eq:conclusion1}
  \sum_{i=1}^n T_{\om_i}T_{\om^i}=\frac12\sum_{i=1}^n(\om_iT_{\om^i}+T_{\om_i}\om^i),
\end{equation}
and
\begin{equation}
  \label{eq:conclusion2}
  \wti\Omega=\Omega-\frac12\sum_{i=1}^n(\om_iT_{\om^i}+T_{\om_i}\om^i)=\Omega -\sum_{i=1}^n T_{\om_i}T_{\om^i}.
\end{equation}
This is the first assertion of the theorem.  For the remainder, 
write out the definition of  $T_{\om_i}$ and $T_{\om^i}$, and use (\ref{innprod}):
\begin{equation}
  \label{eq:sumtomi}
  \sum_{i=1}^n T_{\om_i}T_{\om^i}
=\frac 14\sum_{\al,\beta>0}c_\al c_\beta\langle\al,\beta\rangle t_{s_\al}
  t_{s_\beta}
=
   \frac 14\sum_{\substack{\al>0,\beta>0 \\s_\al(\beta)<0}} 
c_{\al} c_\beta \langle\al,\beta\rangle t_{s_\al} t_{s_\beta},
\end{equation}
with the last equality following as in the proof of Lemma \ref{l:tom}.
\end{proof}

\begin{corollary}\label{c:criterion}  
Retain the setting of Proposition \ref{p:criterion}
but further assume $(\pi,X)$ is irreducible and unitary with
central character $\chi_\nu$ with $\nu \in V$ (as in Definition \ref{d:cc}).  
Let $(\sigma, U)$ be an irreducible representation of $W$ such that
$\Hom_{W}(U,X) \neq 0$.
Then
\begin{equation}\label{ineqeff}
\langle\nu,\nu\rangle\le c(\sigma)
\end{equation}
where
\begin{equation}\label{e:char}
c(\sigma)
=\frac 14\sum_{\al>0}c_\al^2\langle\al,\al\rangle+\frac 14\sum_{\substack{\al>0,\beta>0\\\al\neq\beta, s_\al(\beta)<0}}c_\al c_\beta\langle\al,\beta\rangle \frac {\tr_\sigma(s_\al s_\beta)}{\tr_\sigma(1)}
\end{equation}
is the scalar by which $\Omega_W$ acts in $U$ 
and  $\tr_\sigma$ denotes the character of $\sigma$.
\end{corollary}
\begin{proof}
The result follow from the formula in Theorem \ref{t:Omtilde}
by applying Proposition \ref{p:criterion} to a 
vector $x$ in the $\sigma$ isotypic component of $X$.
\end{proof}

Theorem \ref{t:Omtilde} will play an important role in the
proof of Theorem \ref{t:dirac} below.

\section{The Dirac operator}\label{sec:dirac}

Throughout this section we fix the setting of Section \ref{ss:rs}.

\subsection{The Clifford algebra}
\label{ss:cliff}
Denote by $ C(V^\vee)$ the Clifford algebra defined by $V^\vee$
and $\langle~,~\rangle$. 
More precisely, $
C(V^\vee)$ is the
quotient of the tensor algebra of $V^\vee$ by the ideal generated by
$$\om\otimes \om'+\om'\otimes \om+2\langle \om,\om'\rangle,\quad
\om,\om'\in V^\vee.$$
Equivalently,  $ C(V^\vee)$ is the associative algebra
with unit generated by $V^\vee$ with relations:
\begin{equation}
\om^2=-\langle\om,\om\rangle,\quad \om\om'+\om'\om=-2\langle\om,\om'\rangle.
\end{equation}
Let $\mathsf{O}(V^\vee)$ denote the group of orthogonal transformation of
$V^\vee$ with respect to $\langle~,~\rangle$. This acts by algebra
automorphisms on $ C(V^\vee)$, and the action of $-1\in
\mathsf{O}(V^\vee)$ induces a grading
\begin{equation}
 C(V^\vee)= C(V^\vee)_{\mathsf{even}}+  C(V^\vee)_{\mathsf{odd}}.
\end{equation}
Let $\ep$ be the automorphism of $ C(V^\vee)$ which is $+1$ on $
C(V^\vee)_{\mathsf{even}}$ and $-1$ on $ C(V^\vee)_{\mathsf{odd}}$.
Let ${}^t$ be the transpose antiautomorphism of $ C(V^\vee)$
characterized by
\begin{equation}
\om^t=-\om,\ \om\in V^\vee,\quad (ab)^t=b^ta^t,\ a,b\in C(V^\vee).
\end{equation}
The Pin group  is
\begin{equation}\label{pin}
\mathsf{Pin}(V^\vee)=\{a\in  C(V^\vee)\; |\; \ep(a) V^\vee a^{-1}\subset
V^\vee,~ a^t=a^{-1}\}.
\end{equation}
It sits in a short exact sequence
\begin{equation}\label{ses}
1\longrightarrow \{\pm 1\} \longrightarrow
\mathsf{Pin}(V^\vee)\xrightarrow{\ \ p\ \ } \mathsf{O}(V^\vee)\longrightarrow 1,
\end{equation}
where the projection $p$ is given by $p(a)(\om)=\ep(a)\om a^{-1}$.

We call a simple
$ C[V^\vee]$ module
$(\gamma, S)$ of dimension $2^{[\dim
    V/2]}$ a spin module for $C(V^\vee)$.
When $\dim V$ is even, there is only one such module (up to
equivalence), but if $\dim V$ is odd, there are two inequivalent
spin modules.
We may endow such a module 
with a positive
definite Hermitian form $\langle ~,~\rangle_{ S}$ such that
\begin{equation}
\label{e:sform}
\langle\gamma(a)s,s'\rangle_{ S}=\langle s,\gamma(a^t)\rangle_{ S},\quad\text{for all
  }a\in  C(V^\vee)\text{ and } s,s'\in  S. 
\end{equation}
In all cases,
$(\gamma, S)$ restricts to an irreducible unitary
representation of $\mathsf{Pin}(V^\vee)$.

\subsection{The Dirac operator $D$}

\begin{definition}
\label{d:dirac}
Let  $\{\om_i\}$, $\{\om^i\}$ be dual bases of $V^\vee$,
and recall the elements 
$\wti\om_i\in\bH$ from (\ref{omtilde}).  The abstract
Dirac operator is defined as
\[
\caD = \sum_i \wti \omega_i \otimes \omega_i \in \bH \otimes C(V^\vee).
\]
It is elementary to verify that $\caD$ does not depend on the choice of
dual bases.

Frequently we will work with
a fixed spin module $(\gamma, S)$ for $C(V^\vee)$ 
and a fixed $\bH$-module $(\pi,X)$.  In this
setting, it
will be convenient to define
the Dirac operator for $X$ (and $S$)
as $D=(\pi \otimes \gamma)(\caD)$.
Explicitly,
\begin{equation}\label{dirac}
D=\sum_{i=1}^n \pi(\wti\om_i)\otimes \gamma(\om^i)\in
\operatorname{End}_{\bH\otimes  C(V^\vee)}(X\otimes S).
\end{equation}
\end{definition}

\begin{lemma}
\label{l:adj}
Suppose $X$ is a Hermitian $\bH$-module with invariant form $(~,~)_X$.
With notation as in \eqref{e:sform}, endow $X \otimes S$ with the
Hermitian form $(x\otimes s, x' \otimes s')_{X \otimes S} 
= (x,x')_X \otimes (s,s')_S$.  Then the operator $D$ is self adjoint with
respect to $(~,~)_{X \otimes S}$,
\begin{equation}
( D (x\otimes s), x'\otimes s')_{X\otimes S}=
(x\otimes s,D(x'\otimes s'))_{X\otimes S}
\end{equation}
\end{lemma}
\begin{proof}
This follows from a straight-forward verification.
\end{proof}

We immediately deduce the following analogue of Proposition \ref{p:criterion}.
\begin{proposition}
\label{p:adj}
\label{p:dcriterion}
In the setting of Lemma \ref{l:adj}, a Hermitian $\bH$-module 
is unitary only if
\begin{equation}
\label{eq:dcriterion}
(D^2 (x\otimes s), x\otimes s)_{X\otimes S} \geq 0, \qquad \text{ for all $x\otimes s \in X \otimes S$}.
\end{equation}
\end{proposition}

To be a useful criterion for unitarity, we need to establish
a formula for $D^2$ (Theorem \ref{t:dirac} below).

\subsection{The spin cover $\wti W$}
\label{ss:wtilde}
The Weyl group $W$ acts by orthogonal
transformations on $V^\vee$, and thus is a subgroup of
$\mathsf{O}(V^\vee).$ We define the group $\wti W$ in
$\mathsf{Pin}(V^\vee)$: 
\begin{equation}
\wti W:=p^{-1}(\mathsf{O}(V^\vee))\subset \mathsf{Pin}(V^\vee),\text{ where $p$
  is as in (\ref{ses}).}
\end{equation}
Therefore, $\wti W$ is a central extension of $W$,
\begin{equation}
1\longrightarrow \{\pm 1\}  \longrightarrow
\wti W\xrightarrow{\ \ p\ \ } W\longrightarrow 1.
\end{equation}
We will need a few details about the structure of $\wti W$.
For each $\alpha \in R$, define elements $f_\al \in C(V^\vee)$ via
\begin{equation}
\label{e:fal}
f_\al = \al^\vee / | \al^\vee | \in V^\vee \subset C(V^\vee).
\end{equation}
It follows easily that $p(f_\al) = s_\al$, the
reflection in $W$ through $\al$.  Thus
$\{f_\al \; | \; \al \in R\}$  (or just $\{f_\al \; | \; \al \in \Pi\}$) 
generate $\wti W$.  Obviously $f_\al^2 = -1$.  
Slightly more delicate considerations
 (e.g.~\cite[Theorem 3.2]{Mo}) show that
if  $\al, \beta \in R$, $\gamma = s_\al(\beta)$,
then 
\begin{equation}
\label{e:wrel}
f_\beta f_\al = - f_\al f_\gamma.
\end{equation}

A representation of $\wti W$ is called genuine if it does not factor
to $W$, i.e.~if $-1$ acts nontrivially.  Otherwise it is called
nongenuine.  (Similar terminology applies to $\bC[\wti W]$ modules.)
Via restriction, we can regard a spin module $(\gamma, S)$ for
$C(V^\vee)$ as a unitary $\wti W$ representation.  Clearly it is
genuine.  Since $R^\vee$ spans $V^\vee$, it is also irreducible
(e.g.~\cite{Mo} Theorem 3.3). 
For notational convenience, we lift the $\sgn$ representation of
$W$ to a nongenuine representation of $\wti W$ which we also
denote by $\sgn$.  

We write $\rho$ for the diagonal embedding of $\bC[\wti W]$ into
$\bH \otimes C(V^\vee)$ defined by
extending
\begin{equation}
\label{e:rho}
\rho(\wti w) = t_{p(\wti w)} \otimes \wti w
\end{equation}
linearly.


\begin{lemma} 
\label{l:winvdirac}
Recall the notation of Definition \ref{d:dirac} and \eqref{e:rho}.
For $\wti w\in \wti W$, 
\[
\rho(\wti w) \caD
=\sgn(\wti w) \caD \rho(\wti w)
\]
as elements of $\bH \otimes C(V^\vee)$.
\end{lemma}

\begin{proof} From the definitions and Proposition \ref{p:tilde}(2), we have
\begin{align*}
 \rho(\wti w) \caD
  \rho(\wti w^{-1})
  &=\sum_{i} t_{p(\wti w)} \wti \om_i t_{p(\wti w^{-1})}\otimes
  \wti w \om^i \wti w^{-1}\\
  &=\sum_{i} \wti{p(\wti w)\cdot\omega_i}\otimes
  \wti w \om^i \wti w^{-1}
   \end{align*}
where we have used the $\cdot$ to emphasize the usual action
of $W$ on $S(V^\vee)$.
We argue that in $C(V^\vee)$
\begin{equation}
\label{e:lem}
\wti w \om^i \wti w^{-1}
=\sgn(\wti w) (p(\wti w)\cdot \om^i).
\end{equation}
Then the lemma follows from the
fact that the definition of $\caD$ is independent of the choice of dual bases.

Since $\wti W$ is generated by the various
$f_\al$ for $\al$ simple, it is sufficient to
verify \eqref{e:lem} for $\wti w=f_\al$.  This follows from direct calculation:
$f_\al\om^i f_\al^{-1}=-\frac 1{\langle\al^\vee,\al^\vee\rangle}\al^\vee
\om^i\al^\vee=-\frac 1{\langle\al^\vee,\al^\vee\rangle}
\al^\vee(-\al^\vee\om^i-2\langle\om,\al^\vee\rangle)=-\om^i+(\om,\al)\al^\vee=- s_\al\cdot\om^i.$
\end{proof}

\subsection{A formula for $\caD^2$} 
Set
\begin{equation}\label{omWtilde}
\Omega_{\wti W}=\frac 14\sum_{\substack{\al>0,\beta>0\\s_\al(\beta)<0}}
c_\al c_\beta |\al| |\beta| f_\al f_\beta.
\end{equation}
This is a complex linear combination of elements of $\wti W$, 
i.e.~an element of $\bC[\wti W]$.
Using \eqref{e:wrel}, it is easy to see $\Omega_{\wti
  W}$ is invariant under the conjugation action of $\wti W$.

\begin{theorem}\label{t:dirac}
With notation as in \eqref{casimir},
\eqref{dirac}, \eqref{e:rho}, and \eqref{omWtilde},
\begin{equation}
\caD^2=-\Omega\otimes 1+\rho(\Omega_{\wti W}), 
\end{equation}
as elements of $\bH \otimes C(V^\vee)$.
\end{theorem}

\begin{proof}
It will be useful below to set \begin{equation}
R^2_\circ:=\{(\al,\beta)\in R\times R: \al>0,\beta>0,\al\neq\beta, s_\al(\beta)<0\}.
\end{equation}
To simplify notation, we fix a self-dual (orthonormal) basis
$\{\om_1: ~i=1,\dots, n\}$ of $V^\vee$. From
Definition \ref{d:dirac}, we have
\[
\caD^2=\sum_{i=1}^n\wti\om_i^2\otimes \om_i^2+\sum_{i\neq
  j}\wti\om_i\wti\om_j\otimes \om_i\om_j
\]
Using $\om_i^2 =-1$ and $\om_i\om_j = - \om_i\om_j$ in $C(V^\vee)$ and the notation of
Definition \ref{d:castilde}, we get
\[
\caD^2=-\wti\Omega\otimes 1+
\sum_{i<j}[\wti\om_i,\wti\om_j]\otimes \om_i\om_j.
\]
Applying Theorem \ref{t:Omtilde} to the first term
and Proposition \ref{p:tilde}(3) to the second,
we have
  \begin{align*}
\caD^2
=&-\Omega\otimes 1+
\frac
14\sum_{\al>0}\langle\al,\al\rangle +
\frac
14\sum_{(\al,\beta)\in R^2_\circ} c_\al c_\beta\langle
\al,\beta\rangle t_{s_\al}t_{s_\beta}\otimes 1\\
&-\sum_{i<j}[T_{\om_i},T_{\om_j}]\otimes \om_i\om_j.
\end{align*}
Rewriting $[T_{\om_i},T_{\om_j}]$ using Proposition \ref{p:tilde}(3) ,
this becomes
\begin{align*}
\caD^2=&-\Omega\otimes 1+\frac
14\sum_{\al>0}\langle\al,\al\rangle+
\frac
14\sum_{(\al,\beta)\in R^2_\circ} c_\al c_\beta\langle
\al,\beta\rangle t_{s_\al}t_{s_\beta}\otimes 1\\
&-\frac 14\sum_{i<j}\sum_{(\al,\beta)\in R^2_\circ}c_\al c_\beta
((\om_i,\al)(\om_j,\beta)-(\om_i,\beta)(\om_j,\al))t_{s_\al}t_{s_\beta}\otimes
\om_i\om_j,
\end{align*}
and since $\om_i\om_j = -\om_j\om_i$ in $C(V^\vee)$,
\begin{align*}
\caD^2 &=-\Omega\otimes 1+\frac
14\sum_{\al>0}\langle\al,\al\rangle+\frac
14\sum_{(\al,\beta)\in R^2_\circ} c_\al c_\beta\langle
\al,\beta\rangle t_{s_\al}t_{s_\beta}\otimes 1\\
&-\frac 14\sum_{(\al,\beta)\in R^2_\circ}c_\al
c_\beta{t_{s_\al}}{t_{s_\beta}}\otimes \sum_{i\neq
  j}((\om_i,\al)(\om_j,\beta))\om_i\om_j.
  \end{align*}
  Using \eqref{innprod} and the definition of $f_\al$ in \eqref{e:fal}, we get
  \begin{align*}
\caD^2 &=-\Omega\otimes 1+\frac
14\sum_{\al>0}\langle\al,\al\rangle+\frac
14\sum_{(\al,\beta)\in R^2_\circ} c_\al c_\beta\langle
\al,\beta\rangle {t_{s_\al}}{t_{s_\beta}}\otimes 1\\
&\qquad\qquad\qquad\qquad\qquad\quad \ -\frac 14\sum_{(\al,\beta)\in R^2_\circ}c_\al
c_\beta t_{s_\al}t_{s_\beta}\otimes (|\al||\beta| f_\al
f_\beta+\langle\al,\beta\rangle)\\
&=-\Omega\otimes 1+\frac
14\sum_{\al>0}\langle\al,\al\rangle+\frac
14\sum_{(\al,\beta)\in R^2_\circ} c_\al c_\beta|\al||\beta|t_{p(f_\al)}
t_{p(f_\beta)}\otimes f_\al f_\beta.
\end{align*}
The theorem follows.
\end{proof}

\begin{corollary}\label{c:diracineq}
In the setting of Proposition \ref{p:adj}, assume further that $X$ is irreducible and unitary
with central character $\chi_\nu$ with $\nu \in V$ (as in Definition \ref{d:cc}).
Let
$(\wti\sigma,\wti U)$ be an irreducible representation of $\wti W$
such that $\Hom_{\wti W}(\wti U, X \otimes S) \neq 0$.
Then
\begin{equation}\label{eq:diracineq}
\langle\nu,\nu\rangle\le c(\wti\sigma)
\end{equation}
where
\begin{equation}\label{eq:sigmaWtilde}
c(\wti\sigma)
=\frac
14\sum_{\al>0}c_\al^2\langle\al,\al\rangle+\frac 14\sum_{\substack{\al>0,\beta>0,\al\neq\beta\\s_\al(\beta)<0}}
c_\al c_\beta |\al||\beta|\frac{\tr_{\wti\sigma}(f_\al f_\beta)}{\tr_{\wti\sigma}(1)},
\end{equation}   
is the scalar by which $\Omega_{\wti W}$ acts in $\wti U$ 
and $\tr_{\wti \sigma}$ denotes the character of $\wti \sigma$.
\end{corollary}

\begin{proof}
The corollary follows by applying Proposition \ref{p:dcriterion}
to a vector $x \otimes s$ in the $\wti \sigma$ isotypic component
of $X \otimes S$, and then using
the formula for $D^2 = (\pi \otimes \gamma)(\caD^2)$ from Theorem \ref{t:dirac}
and the formula for $\pi(\Omega)$ from Lemma \ref{scalar}.
\end{proof}

\section{dirac cohomology and vogan's conjecture}
\label{s:vogan}
Suppose $(\pi, X)$ is an irreducible $\bH$ module with central character 
$\chi_\nu$.  By Lemma \ref{l:winvdirac}, the kernel of the Dirac operator on $X \otimes S$ is 
invariant under $\wti W$.  Suppose $\ker(D)$ is nonzero and that
$\wti \sigma$ is an irreducible representation of $\wti W$ appearing in $\ker(D)$.  
Then in the notation of Corollary \ref{c:diracineq}, Theorem \ref{t:dirac} and Lemma
\ref{scalar} imply that
\[
\langle \nu, \nu \rangle = c(\wti\sigma).
\]
In particular, the length of $\nu$ is determined by the $\wti W$
structure of $\ker(D)$.  Theorem \ref{t:vogan} below 
says that $\chi_\nu$ itself is determined by this information.

In this section (for the reasons mentioned in the introduction), we
fix a crystallographic root system $\Phi$ and set the parameter function
$c$ in Definition \ref{d:graded} to be identically 1,
i.e.~$c_\al=1$ for all $\al\in
R.$

\subsection{Geometry 
of irreducible representations
of $\wti W$.}
\label{ss:ciubo}
Let $\fg$ denote the complex semisimple Lie algebra corresponding to $\Phi$.
In particular, $\frg$ has a Cartan subalgebra $\frh$ such that
$\frh \simeq V$ canonically. 
Write $\caN$ for the nilpotent cone in $\fg$. 
Let $G$ denote the adjoint group $\mathrm{Ad}(\frg)$ acting by the adjoint
action on $\caN$. 

Given $e \in \caN$, let $\{e,  h,f\} \subset \frg$ denote an
$\mathfrak{s}
\mathfrak{l}_2$ triple with $h \in \frh$ semisimple.
Set
\begin{equation}
\label{e:nu}
\nu_e = \frac 12 h\in \frh \simeq V.
\end{equation}
The element $\nu_e$ depends on the choices involved. But its $W$-orbit
(and in particular $\langle \nu_e, \nu_e \rangle$ and the central character
$\chi_\nu$ of Definition \ref{d:cc}) are well-defined 
independent of the $G$ orbit of $e$.

Let 
\begin{equation}
\label{e:Nsol}
\caN_\sol =
\left  \{ e \in \caN \; | \; \text{ the centralizer of $e$ in $\frg$
is a solvable Lie algebra} \right \}.
\end{equation}
Then $G$ also acts on $\caN_\sol$.

Next
let  $A(e)$ denote the component group of the centralizer
of $e \in \caN$ in $G$.
To each $e \in \caN$,
Springer has defined a graded representation of $W \times A(e)$
(depending only on the $G$ orbit of $e$)
on the total cohomology $H^\bullet(\CB^e)$ of the Springer fiber
over $e$.  Set $d(e) = 2\dim(\CB^e)$, and define
\begin{equation}
\label{e:springer}
\sigma_{e,\phi} = \left(H^{d(e)}(\CB^e)\right)^\phi \in \Irr(W) \cup \{0\},
\end{equation}
the $\phi$ invariants in the top degree.  (In
general, given a finite group $H$, we write $\Irr(H)$ for the set of
equivalence classes of its irreducible representations.)
Let $\Irr_0(A(e)) \subset
\Irr(A(e))$ denote the subset of representations of ``Springer type'', i.e.~those $\phi$ such that $\sigma_{e,\phi} \neq 0$.

Finally, let $\Irr_\gen(\wti W) \subset \Irr(\wti W)$ denote the subset of
genuine representations.

\begin{theorem}[{\cite{ciubo:weyl}}]
\label{t:class}
Recall the notation of \eqref{e:nu},  \eqref{e:Nsol}, and \eqref{e:springer}.
Then
there is a surjective map
\begin{equation}\label{e:class}
\Psi:~\Irr_\gen(\wti W)
\longrightarrow 
G \bs \caN_\sol
\end{equation}
with the
following properties:
\begin{enumerate}
\item If $\Psi(\wti\sigma)=G\cdot e$, then 
\begin{equation}
c(\wti\sigma)=\langle \nu_e,\nu_e\rangle,
\end{equation}
where $c(\wti \sigma)$ is defined in \eqref{eq:sigmaWtilde}.

\item Fix a spin module $(\gamma,S)$ for $C(V^\vee)$.  
\begin{enumerate}
\item[(a)]
If $e \in \caN_\sol$ and $\phi \in \Irr_0(A(e))$, then there exists
$\wti \sigma \in \Psi^{-1}(G\cdot e)$ so that
\[
\Hom_W\left (\sigma_{e,\phi}, \wti \sigma \otimes S \right) \neq 0.
\]

\item[(b)]
If $\Psi(\wti \sigma) = G\cdot e$, then there exists
$\phi \in \Irr_0(A(e))$ such that
\[
\Hom_W\left (\sigma_{e,\phi}, \wti \sigma \otimes S \right) \neq 0.
\]
\end{enumerate}
\end{enumerate}
\end{theorem}

Together with Corollary \ref{c:diracineq}, 
we immediately obtain the following.

\begin{corollary}
\label{c:unip}
Suppose $(\pi,X)$ is an irreducible unitary $\bH$-module with central
character $\chi_\nu$ with $\nu \in V$ (as in Definition \ref{d:cc}).  Fix a spin module $(\gamma,S)$
for $C(V^\vee)$.
\begin{enumerate}
\item[(a)]
Let $(\wti\sigma,\wti U)$ be a representation of $\wti W$
such that $\Hom_{\wti W}\left ( \wti U, X\otimes  S\right )\neq 0$. 
 In the notation of Theorem \ref{t:class}, write 
$\Psi(\wti \sigma) = G\cdot e$.
Then
\begin{equation}\label{eq:diracineq2}
\langle\nu,\nu\rangle\le \langle\nu_e,\nu_e\rangle.
\end{equation}

\item[(b)]
Suppose $e\in \caN_\sol$ and $\phi \in \Irr_0(A(e))$ such that
$\Hom_W(\sigma_{(e,\phi)}, X) \neq 0$.
Then
\begin{equation}
\langle\nu,\nu\rangle\le \langle\nu_e,\nu_e\rangle.
\end{equation}
\end{enumerate}
\end{corollary}

\begin{remark} 
\label{r:unip}
The bounds in Corollary \ref{c:unip} represents the best possible in
the sense that there exist $X$ such that the inequalities are actually
equalities.  For example, consider part (b) of the corollary, fix
$\phi \in \Irr_0(A(e))$, and let $X_{t}(e,\phi)$ be the unique
tempered representation of $\bH$ parametrized by $(e,\phi)$
in the Kazhdan-Lusztig classification (\cite{KL}).  Thus
$X_{t}(e,\phi)$ is an irreducible unitary representation with central
character $\chi_{\nu_e}$ and, as a representation of $W$,
\[
X_{t}(e,\phi) \simeq H^\bullet(\C B_e)^\phi.
\]
In particular, $\sigma_{e,\phi}$ occurs
with multiplicity one in $X_{t}(e,\phi)$ (in the top degree).
Thus the inequality in Corollary \ref{c:unip}(b) applied to $X_t(e,\phi)$ is an equality.

The representations $X_t(e,\phi)$ 
will play an important role in 
our proof of Theorem \ref{t:vogan}.
\end{remark}

\subsection{Applications to unitary representations}\label{s:unit}
Recall that there exists a unique open dense $G$-orbit in $\C N$, the
regular orbit; let $\{e_\mr,h_\mr,f_\mr\}$ be a corresponding
$\mathfrak{sl}_2$ with $h_\mr \in \frh$, and set $\nu_\mr = \frac12 h_\mr$.
If $\fg$ is simple, then there exists a unique open
dense $G$-orbit in the complement of $G\cdot e_\mr$ in $\C N$
called the subregular
orbit. Let $\{e_\msr,h_\msr,f_\msr\}$ be an $\mathfrak{sl}_2$ triple
for the subregular orbit with $h_\msr \in \frh$, and set 
$\nu_\msr = \frac 12 h_\msr$.

The tempered module $X_t(e_\mr,\triv)$ is the Steinberg
discrete series, and we have $X_t(e_\mr,\triv)|_W=\sgn.$ When $\fg$
is simple, the tempered module $X_t(e_\msr,\triv)$ has dimension $\dim
V+1$, and $X_t(e_\msr,\triv)|_W=\sgn\oplus\refl,$ where $\refl$ is the
reflection $W$-type.

Now we can state certain bounds for unitary $\bH$-modules.

\begin{corollary}\label{c:unit} Let $(\pi,X)$ be an irreducible unitary $\bH$-module with central
character $\chi_\nu$ with $\nu \in V$ (as in Definition
\ref{d:cc}). Then, we have:
\begin{enumerate}
\item $\langle\nu,\nu\rangle\le\langle\nu_\mr,\nu_\mr\rangle$;
\item if $\fg$ is simple of rank at least $2$, and $X$ is \emph{not} the trivial or the
  Steinberg $\bH$-module, then $\langle\nu,\nu\rangle\le\langle\nu_\msr,\nu_\msr\rangle$.
\end{enumerate}
\end{corollary}

\begin{proof} The first claim follows from Corollary \ref{c:unip}(1)
  since $\langle\nu_e,\nu_e\rangle\le\langle\nu_\mr,\nu_\mr\rangle$, for every $e\in\C N.$

 For the second claim,
  assume $X$ is not the trivial or the Steinberg $\bH$-module. Then
  $X$ contains a $W$-type $\sigma$ such that
  $\sigma\neq\triv,\sgn$. We claim that $\sigma\otimes S$, where $S$
  is a fixed irreducible spin module, contains a
  $\wti W$-type $\wti\sigma$ which is not a spin module. If this
were not the case, assuming for simplicity that $\dim V$ is even, we
would find that $\sigma\otimes S=S\oplus\dots\oplus S$, where there are
$\tr_\sigma(1)$ copies of $S$ in the right hand side. In particular,
we would get $\tr_\sigma(s_\al s_\beta)\tr_S(f_\al
f_\beta)=\tr_\sigma(1)\tr_S(f_\al f_\beta)$. Notice that this formula
is true when $\dim V$ is odd too, since the two inequivalent
spin modules in this case have characters which have the same
value on  $f_\al f_\beta.$ If $\langle\al,\beta\rangle\neq 0$, then
we know that $\tr_S(f_\al f_\beta)\neq 0$ (\cite{Mo}). This means that
$\tr_\sigma(s_\al s_\beta)=\tr_\sigma(1)$, for all non-orthogonal
roots $\al,\beta$. One verifies directly that, when $\Phi$ is simple
of rank two, this relation does not hold. Thus we obtain a contradiction.

Returning to the second claim in the corollary, let $\wti\sigma$ be a
$\wti W$-type appearing in 
$X\otimes S$ which is not a spin module. Let $e$ be a nonregular nilpotent element such that
$\Psi(\wti\sigma)=G\cdot e.$ Corollary \ref{c:unip} says that
$\langle\nu,\nu\rangle\le\langle\nu_e,\nu_e\rangle.$ To complete the
proof, recall that if $\fg$ is simple, the largest value for
$\langle\nu_e,\nu_e\rangle$, when $e$ is not a regular element, is
obtained when $e$ is a subregular nilpotent element.
\end{proof}

\begin{remark}
\label{r:t}
Standard considerations for reducibility of principal series
allow one to deduce a strengthened version  of Corollary \ref{c:unit}(1),
namely that $\nu$ is contained in the convex hull of the Weyl group orbit of $\nu_\mr$.
We also note that Corollary \ref{c:unit}(2) implies, in particular,
that the trivial $\bH$-module is isolated
in the unitary dual of $\bH$ for all simple root systems of rank at
least two.   (Sometimes this is called Kazhdan's Property T.)  
\end{remark}

\begin{remark}
\label{r:unit}
There is another, subtler application of Corollary
\ref{c:unip}(2). Assume that $X(s,e,\psi)$ is an irreducible $\bH$-module
parametrized in the Kazhdan-Lusztig classification by the
$G$-conjugacy class of $\{s,e,\psi\}$, where $s\in \fh_0\cong V_0$, $[s,e]=e$,
$\psi\in\Irr_0 A(s,e)$.  The group $A(s,e)$ embeds canonically in
$A(e)$.  Let $\Irr_0 A(s,e)$ denote the subset of elements
in $\Irr A(s,e)$ which
appear in the restriction of an element of $\Irr_0 A(e)$. 
The module $X(s,e,\psi)$
is characterized by the property that it contains every $W$-type
$\sigma_{(e,\phi)}$, $\phi\in \Irr_0 A(e)$ such that
$\Hom_{A(s,e)}(\psi,\phi)\neq 0$. 

Let $\{e,h,f\}$ be an $\mathfrak{sl}_2$
triple containing $e$. One may choose $s$ such that $s=\frac 12
h+s_z$, where $s_z\in V_0$ centralizes $\{e,h,f\}$ and $s_z$ is orthogonal to
$h$ with respect to $\langle~,~\rangle$. When $s_z=0$, we have
$A(s,e)=A(h,e)=A(e)$, and 
$X(\frac 12 h,e,\psi)$ is the tempered module $X_t(e,\phi)$ ($\phi=\psi$) from
before. 

Corollary \ref{c:unip}(2) implies that if $e\in\C N_{\mathsf{sol}}$,
then $X(s,e,\psi)$ is unitary \emph{if and only if} $X(s,e,\psi)$ is tempered.
\end{remark}

\subsection{Dirac cohomology and Vogan's Conjecture}
\label{ss:diraccoh}

As discussed above, Vogan's Conjecture suggests that
for an irreducible unitary representation $X$, the $\wti W$ structure of
the kernel of the Dirac operator $D$ should determine the
infinitesimal character of $X$.  
This is certainly 
false for nonunitary representations.
But since it is difficult to imagine a proof of an algebraic statement which
applies only to unitary representations,
we use an idea of Vogan and 
enlarge the class of irreducible unitary representations for which $\ker(D)$
is nonzero to the class of representations with nonzero Dirac
cohomology in the following sense.

\begin{definition}
\label{d:dcoh}
In the setting of Definition \ref{d:dirac}, define
\begin{equation}
H^D(X):=\ker D\big / \left(\ker D\cap \Im D\right )
\end{equation}
and call it the Dirac cohomology of $X$.
(For example, if $X$ is unitary, Lemma \ref{l:adj} implies
$\ker(D) \cap \Im (D) = 0$, and so $H^D(X) = \ker (D)$.)
\end{definition}

Our main result is as follows.

\begin{theorem}
\label{t:vogan}
Let $\bH$ be the graded affine Hecke algebra attached to
a crystallographic root system $\Phi$ and constant parameter
function $c \equiv 1$ (Definition \ref{d:graded}). 
Suppose $(\pi,X)$ is an $\bH$ module with central character $\chi_\nu$
with $\nu \in V$ (as in Definition \ref{d:cc}).
In the setting of Definition \ref{d:dcoh}, suppose 
that $H^D(X) \neq 0$.
Let $(\wti\sigma,\wti U)$ be a 
representation of $\wti W$ such that $\Hom_{\wti W}(\wti U, H^D(X) ) \neq 0$.
Using Theorem \ref{t:class}, write $\Psi(\wti \sigma) = G\cdot e$.
Then
\[
\chi_\nu = \chi_{\nu_e}.
\]
\end{theorem}

We roughly follow the outline of the proof in the real case
proposed by Vogan in \cite[Lecture 3]{v:notes}
and completed in \cite{HP}.

\begin{proposition}
\label{p:hp} 
Let $(\pi, X)$ be an irreducible $\bH$ module with central
character $\chi_\nu$ with $\nu \in V$ (as in Definition \ref{d:cc}).
In the setting of Definition \ref{d:dirac}, 
suppose $(\wti \sigma, \wti U)$ is an irreducible
representation of $\wti W$ such that
$\Hom_{\wti W} (\wti U, X \otimes S) \neq 0$.
Write $\Psi(\wti \sigma) = G\cdot e$ as in Theorem \ref{t:class}.
Assume further that
$\langle\nu,\nu\rangle=\langle\nu_e,\nu_e\rangle$. Then
\[
\Hom_{\wti W}\left (\wti U, H^D(X)\right) \neq 0.
\]
\end{proposition}

\begin{proof} Let $x\otimes s$ be an element of the $\wti \sigma$
isotypic component of $X\otimes  S$.
By Theorem \ref{t:dirac}
  and Theorem \ref{t:class}, have
\begin{equation}
\label{e:hp}
D^2(x\otimes
s)=\left (-\langle\nu,\nu\rangle+\langle\nu_e,\nu_e\rangle
\right )(x\otimes s)=0.
\end{equation}
Since $X$ is unitary, $\ker D\cap\Im D=0$,
and so \eqref{e:hp} implies 
$x\otimes s \in \ker(D) = H^D(X).$
\end{proof}

As in setting of real groups, Theorem \ref{t:vogan}
can be deduced from a purely algebraic statement 
(c.f.~Theorem 2.5 and Corollary 3.5 in \cite{HP}).

\begin{theorem}
\label{t:hp}
Let $z\in Z(\bH)$ be given. Then there  exist
 $a, b\in \bH\otimes C(V^\vee)$ 
and a unique element $\zeta(z)$ in
the center of $\bC[\wti W]$
such that
$$z\otimes 1 =\rho(\zeta(z))+\caD a+b\caD $$
as elements in  $\bH \otimes C(V^\vee)$.
\end{theorem}

\begin{proposition}
\label{p:hpv}
Theorem \ref{t:hp} implies Theorem \ref{t:vogan}.
\end{proposition}

\begin{proof}
In the setting of Theorem \ref{t:vogan}, 
suppose $\Hom_{\wti W}(\wti U , H^D(X)) \neq 0$.  
Then there exist $\wti x=x\otimes s\neq 0$
in the $\wti \sigma$  isotypic component of $X\otimes
  S$ such that $\wti x\in \ker D\setminus \Im D.$ 
Then for every
  $z\in Z(\bH)$, we have
\[
(\pi(z)\otimes 1)\wti x=\chi_\nu(z) \wti x
\]
and 
\[
(\pi \otimes \gamma)(\rho(\zeta(z))) \wti
  x=\wti\sigma(\zeta(z))\wti x.
\]
Note that the right-hand sides of the previous two displayed
equations are scalar multiples of $\wti x$.
Assuming Theorem \ref{t:hp}, we have 
\begin{equation}
\left (\pi(z)\otimes 1-(\pi \otimes  \gamma)\rho(\zeta(z))\right )\wti x=(Da+bD)\wti x=Da\wti x,
\end{equation}
which would imply that $\wti x\in \Im D$, unless $Da\wti x=0.$ So we
must have $Da\wti x=0$, and therefore 
\begin{equation}\label{eq:tech}
\chi_\nu(z)=\wti\sigma(\zeta(z)), \text{ for all }z\in Z(\bH). 
\end{equation}
The statement of Theorem \ref{t:vogan} (and hence
of the current proposition) will follow if we can show
 $\wti\sigma(\zeta(z))=\chi_{\nu_e}(z)$ for
all $z\in Z(\bH)$
where $\Psi(\wti \sigma) = G\cdot e$ as in Theorem \ref{t:class}. 

Using Theorem \ref{t:class}(2b), choose $\phi \in \Irr_0(A(e))$
such that
\begin{equation}
\label{eq:tech2}
\Hom_{\wti W}(\wti \sigma, \sigma_{e,\phi} \otimes S) \neq 0,
\end{equation}
and consider the unitary $\bH$ module $X(e,\phi)$ of
Remark \ref{r:unip} with central character $\chi_{\nu_e}$.  
Then since $X_t(e,\phi)$ contains the $W$ type $\sigma_{e,\phi}$,
\eqref{eq:tech2} implies
\[
\Hom_{\wti W}(\wti U, X \otimes S) \neq 0.
\]
So Proposition \ref{p:hp} implies that 
\[
\Hom_{\wti W}(\wti U, H^D(X)) \neq 0.
\]
Since $X(e,\phi)$ has central character $\chi_{\nu_e}$,
\eqref{eq:tech} applies to give
$\chi_{\nu_e}(z)= \wti\sigma(\zeta(z))$ for all
$z\in \bH$.  This completes the proof.
\end{proof}

\section{proof of Theorem \ref{t:hp}}
\label{s:proof}

In this section, we let $\bH$ be defined by
an arbitrary root system $\Phi$ and arbitrary parameter
function $c$.  All the results
below (including the proof of Theorem \ref{t:hp}) 
hold in this generality.

Motivated by \cite[Section 3]{HP}, we define
\begin{equation}
\label{e:d}
d
\; : \; 
\bH \otimes C(V^\vee) \longrightarrow \bH\otimes C(V^\vee).
\end{equation}
on a simple tensor of the form $ a= h \otimes v_1\cdots v_k$
(with $h \in H$ and $v_i \in V^\vee$) via
\[
d(a) = \caD a - (-1)^ka \caD,
\]
and extend linearly to all of $\bH \otimes C(V^\vee)$.

Then Lemma \ref{l:winvdirac} implies that $d$ interchanges
the spaces
\begin{equation}
\label{e:winv}
(\bH \otimes C(V^\vee))^{\triv} = \{a \in \bH \otimes C(V^\vee) 
\; | \; \rho(\wti w) a = a \rho(\wti w) \}
\end{equation}
and
\begin{equation}
\label{e:winvsgn}
(\bH \otimes C(V^\vee))^{\sgn} = \{a \in \bH \otimes C(V^\vee) 
\; | \; \rho(\wti w) a = \sgn(\wti w) a \rho(\wti w) \}.
\end{equation}
(Such complications are not encountered in \cite{HP} since
the underlying real group is assumed to be connected.)
Let $d^{\triv}$ (resp.~$d^{\sgn}$) 
denote the restriction of $d$ to the space
in \eqref{e:winv} (resp.~\eqref{e:winvsgn}).
We will deduce Theorem \ref{t:hp} from the following.

\begin{theorem}
\label{t:hp2}
With notation as in the previous paragraph,
\[
\ker(d^{\triv}) = \Im(d^{\sgn}) \oplus \rho(\bC[\wti W]^{\wti W}).
\]
\end{theorem}

To see that Theorem \ref{t:hp2} implies Theorem \ref{t:hp}, take
$z\in Z(\bH)$.  Since $z\otimes 1$ is in $(\bH \otimes C_\even(V^\vee))^{\triv}$
and clearly commutes with $\caD$,
$z \otimes 1$ is in the kernel of $d^{\triv}$.  So the conclusion of
Theorem \ref{t:hp2} implies $z \otimes 1 = d^{\sgn}(a) + \rho(\zeta(z))$
for a unique $\zeta(z) \in \bC[\wti W]^{\wti W}$ and an element $a$ of
$(\bH \otimes C_\odd(V^\vee))^{\sgn}$.  In particular
$d^{\sgn}(a) = \caD a + a \caD$.  
Thus 
\[
z \otimes 1 = \rho(\zeta(z)) + \caD a + a \caD,
\]
in fact a slightly stronger conclusion than that of Theorem \ref{t:hp}.

Thus everything comes down to proving Theorem \ref{t:hp2}.  The remainder
of this section is devoted to doing so.  We being
with some preliminaries.

\begin{lemma}
\label{l:inclusion}
We have 
\[
\rho(\bC[\wti W]^{\wti W}) \subset
\ker(d^{\triv}).
\]
\end{lemma}
\begin{proof}
Fix $\wti w \in \wti W$ and let $s_{\alpha_1} \cdots s_{\alpha_k}$
be a reduced expression of $p(w)$ with $\alpha_i$ simple.  Then 
(after possibly replacing $\alpha_1$ with $-\alpha_1$), $\wti w = f_{\alpha_1}
\cdots f_{\alpha_k}$.  
Set $a = \rho(\wti w)$.  Then the definition of $d$ and
Lemma \ref{l:winvdirac} imply
\begin{align*}
d(a) = \caD a -(1)^k a \caD &= (1 - (-1)^k\sgn(\wti w))\caD a \\
&=(1 - (-1)^k (-1)^k)\caD a = 0,
\end{align*}
as claimed.
\end{proof}

\begin{lemma}
\label{l:d2}
We have $(d^{\triv})^2 = (d^{\sgn})^2= 0$.
\end{lemma}
\begin{proof}
For any $a \in \bH \otimes C(V^\vee)$, one computes directly from the
definition of $d$ to find
\[
d^2(a) = \caD^2 a - a \caD^2.
\]
By Theorem \ref{t:dirac}, $\caD^2 = - \Omega \otimes 1 + \rho(\Omega_{\wti W})$.
By Lemma \ref{l:cascentral}, $-\Omega \otimes 1$ automatically commutes
with $a$.  If we further assume that $a$ is in $(\bH \otimes C(V^\vee))^{\triv}$,
then $a$ commutes with $\rho(\Omega_{\wti W}$) as well.  
Since each term in the definition $\Omega_{\wti W}$ is in the kernel of $\sgn$,
the same conclusion holds if $a$ is in $(\bH \otimes C(V^\vee))^{\sgn}$.
The lemma follows.  
\end{proof}

\medskip

We next introduce certain graded objects (as in the approach of \cite[Section 4]{HP}).
Let $S^j(V^\vee)$
denote the subspace of elements of degree  $j$
in $S(V^\vee)$.  Let $\bH^j$ denote the subspace of $\bH$ consisting of products
elements in the image of $\bC[W]$ and  $S^j(V^\vee)$ under the maps
described in (1) and (2) in Definition \ref{d:graded}.
Then it is easy to check (using \eqref{hecke}) that
 $\bH^0 \subset \bH^1 \subset \cdots$ is an algebra filtration.
Set $\ol \bH^j = \bH^j/\bH^{j-1}$ and let
$\ol \bH = \bigoplus_j\ol\bH^j $ denote the associated graded algebra.
Then $\ol \bH$ identifies with  $\bC[W] \rtimes S(V^\vee)$ 
with $\bC[W]$ acting in 
natural way:
\[
t_w \omega t_{w^{-1}} = w(\omega).
\]
We will invoke these identification often without comment.
Note that $\ol \bH$ does not depend on the parameter function $c$
used to define $\bH$.

The map $d$ of \eqref{e:d} induces a map
\begin{equation}
\ol d
\; : \; 
\ol \bH \otimes C(V^\vee) \longrightarrow \ol \bH\otimes C(V^\vee).
\end{equation}
Explicitly, if we fix a self-dual basis $\{\om_1, \dots, \om_n\}$ of $V^\vee$,
then the value of $\ol d$ on a simple tensor of the form $ a
= t_w f \otimes v_1\cdots v_k$
(with $t_w f \in \bC[W] \ltimes S(V^\vee)$ 
and $v_i \in V^\vee$) is given by
\begin{equation}
\label{e:dbar}
\begin{aligned}
\ol d(a) &= 
\sum_i \omega_i t_w f \otimes \omega_i v_1\cdots v_k 
- (-1)^k\sum_i t_w f\omega_i \otimes v_1 \cdots v_k \omega_i \\
&=
\sum_i t_ww^{-1}(\omega_i) f\otimes \omega_i v_1\cdots v_k 
- (-1)^k\sum_i t_w f\omega_i \otimes v_1 \cdots v_k \omega_i.
\end{aligned}
\end{equation}
We will deduce Theorem \ref{t:hp2}
from the computation of the cohomology of $\ol d$.  We need some final preliminaries.

\begin{lemma}
\label{l:oddder}
The map $\ol d$ of \eqref{e:dbar} is an odd derivation in the sense that if
$a= t_w f \otimes v_1\cdots v_k\in  \ol \bH$ and $b \in \ol \bH$
is arbitrary, then
\[
\ol d(a b) = \ol d(a) b  + (-1)^k a\ol d(b).
\]
\end{lemma}
\begin{proof}
Fix $a$ as in the statement of the lemma.
A simple induction reduces the general case of the
lemma to the following three special cases:
(i) $b = \omega \otimes 1$ for $\omega \in V^\vee$; 
(ii) $b= 1 \otimes \omega$ for $\omega \in V^\vee$; 
and (iii) $b= t_s \otimes 1$ for $s = s_\al$ a simple
reflection in $W$.  (The point is that
these three types of elements generate $\ol \bH$.)  
Each of these cases follows from a straight-forward 
verification.  For example, consider the first case, 
$b = \omega \otimes 1$.  Then from the definition of $\ol d$, we
have
\begin{equation}
\label{e:dbar1}
\ol d(ab) = 
\sum_i\omega_i t_w f \omega \otimes \omega_1v_1\cdots v_k
- (-1)^k \sum_it_w f \omega \omega_i \otimes v_1\cdots v_k \omega_i.
\end{equation}
On the other hand, since it is easy to see that
$d(b) = 0$ in this case, we have
\begin{equation}
\label{e:dbar2}
\begin{aligned}
\ol d(a) b + (-1)^k a\ol d(b) &= \ol d(a) b \\
&=  \sum_i\omega_i t_w f \omega \otimes \omega_1v_1\cdots v_k
- (-1)^k \sum_it_w f \omega_i \omega \otimes v_1\cdots v_k \omega_i.
\end{aligned}
\end{equation}
Since $S(V^\vee)$ is commutative, \eqref{e:dbar1} and \eqref{e:dbar2}
coincide, and the lemma holds in this case.  The other two remaining
cases hold by similar direct calculation.  We omit the details.
\end{proof}

\begin{lemma}
\label{l:dbar2}
The map $\bar d$ satisfies $\bar d^2 = 0$.
\end{lemma}

\begin{proof}
Fix $a= t_w f \otimes v_1\cdots v_k\in  \ol \bH$ 
and let $b$ be arbitrary.
Using Lemma \ref{l:oddder}, one computes directly from the
definitions to find
\[
\ol d^2(ab) =\ol d^2(a)b + a\ol d^2(b).
\]
It follows that to establish the
current lemma in general, it suffices to check that $d^2(b) = 0$ for
each of the three kinds of generators $b$ appearing in the proof
of Lemma \ref{l:oddder}.  Once again this is a straight-forward verification
whose details we omit.  (Only case (iii) is nontrivial.)
\end{proof}

\begin{lemma}
\label{l:winc}
Let $\ol \rho$ denote the diagonal embedding of $\bC[\wti W]$
in $\ol \bH \otimes C(V^\vee)$ defined by linearly extending
\[
\ol \rho (\wti w) = t_{p(f_\alpha)} \otimes \wti w
\]
for $\wti w \in \wti W$.
Then
\[
\ol \rho(\bC[\wti W]) \subset \ker(\ol d).
\]
\end{lemma}
\begin{proof}
As noted in Section \ref{ss:wtilde}, the various $f_\al = \al^\vee/|\al^\vee|$ 
(for $\alpha$ simple) generate $\wti W$.   Furthermore $p(f_\al) = s_\al$.
So Lemma \ref{l:oddder} implies that the current lemma will follow if we
can prove
\[
\ol d(t_{s_\al} \otimes \al^\vee) = 0
\]
for each simple $\al$.
For this we compute directly,
\begin{align*}
\ol d(t_{s_\al} \otimes \al^\vee) 
&= \sum_i \omega_i t_{s_\al} \otimes \omega_i\alpha^\vee
+\sum_i t_{s_\al}\omega_i \otimes \alpha^\vee\omega_i \\
&= \sum_i t_{s_\al} s_\al(\omega_i)\otimes \omega_i\alpha^\vee
+\sum_i t_{s_\al}\omega_i \otimes \alpha^\vee\omega_i \\
&= \sum_i t_{s_\al} (\omega_i - (\al,\omega_i))\al^\vee
\otimes \omega_i\alpha^\vee
+\sum_i t_{s_\al}\omega_i \otimes \alpha^\vee\omega_i \\
&= \sum_i t_{s_\al} \omega_i \otimes (\omega_i\al^\vee +\al^\vee\omega_i)
- \sum_it_{s_\al}(\al,\omega_i)\al^\vee
\otimes \omega_i\alpha^\vee \\
&= -2\sum_i t_{s_\al} \omega_i \otimes \langle \al^\vee, \omega_i \rangle
- \sum_it_{s_\al}\al^\vee
\otimes (\al,\omega_i)\omega_i\alpha^\vee \\
&= -2\sum_i t_{s_\al} \langle \al^\vee, \omega_i \rangle \omega_i \otimes 1
- \sum_it_{s_\al}\al^\vee
\otimes \frac{\langle \omega_i,\al\rangle}{\langle\al^\vee, \al^\vee\rangle}
\omega_i\alpha^\vee \\
&=
-2t_{s_\al}\al^\vee \otimes 1 \quad - \quad t_{s_\al}\al^\vee \otimes 
\frac{(\al^\vee)^2}{\langle\al^\vee, \al^\vee\rangle}\\
&= -2t_{s_\al}\al^\vee \otimes 1 \quad + \quad 2t_{s_\al}\al^\vee \otimes 1 =0.
\end{align*}
\end{proof}

\medskip

Note 
from \eqref{e:dbar}, it follows that $\ol d$
preserves the subspace $S(V^\vee) \otimes C(V^\vee)
 \subset \ol \bH\otimes
C(V^\vee)$.  Write $\ol d'$ for the restriction of $\ol d$
to $S(V^\vee) \otimes C(V^\vee)$.

\begin{lemma}
\label{l:untwisted}
With notation as in the previous paragraph,
\[
\ker(\ol d') = \Im(\ol d') \oplus \bC(1\otimes 1).
\]
\end{lemma}
\begin{proof}
An elementary calculation shows that $\ol d'$ is a multiple of the differential in the Koszul complex whose cohomology is well-known. (See \cite [Lemma 4.1]{HP}, for instance.)
\end{proof}

We can now assemble these lemmas into the computation
of the cohomology of $\ol d$.

\begin{proposition}
\label{p:dbar}
We have
\[
\ker(\ol d) = \Im(\ol d) \oplus \ol\rho(\bC[\wti W]).
\]
\end{proposition}
\begin{proof}
By Lemmas \ref{l:dbar2} and \ref{l:winc},
$\Im(\ol d) +  \ol\rho(\bC[\wti W]) \subset \ker(\ol d)$.
Since it follows from the definition of 
$\ol d$ that $\Im(\ol d)$ and $\ol \rho(\bC[\wti W])$ intersect trivially,
we need only establish the reverse inclusion.  
Fix $a \in \ker(\ol d)$ and write it as a sum of simple tensors
of the form 
$t_{w} f \otimes v_1\cdots v_k$.  For each $w_j \in W$, 
let $a_j$ denote the sum of the simple tensors appearing
in this expression for $a$ which have $t_{w_j}$ in them.
Thus $a= a_1 + \cdots a_l$,
and we can arrange the indexing so that each $a_i$ is nonzero.
Since $\ol d(a) = 0$, 
\begin{equation}
\label{e:a}
\ol d(a_1) + \cdots +\ol d(a_l) = 0.  
\end{equation}
Each term $\ol d(a_i)$ is a sum of simple tensors of the form
$t_{w_i} f \otimes v_1\cdots v_k$.
Since the $w_i$ are distinct, the only way \eqref{e:a}
can hold is if each $\ol d(a_i) = 0$.
Choose $\wti w_i \in \wti W$ such that $p(\wti w_i) = w_i$.  Set
\[
a_i' = \ol \rho(\wti w_1^{-1}) a_i \in S(V) \otimes C(V^\vee).
\]
Using Lemmas \ref{l:oddder}
and \ref{l:winc}, we  have
\[
\ol \rho (\wti w_i) \ol d(a_i')= \ol d(\ol \rho(\wti w_i) a_i) = \ol d(a_i) = 0.
\]
Thus for each $i$,
\[
\ol d(a'_i) = 0.
\]
Since each $a_i' \in S(V^\vee) \otimes C(V^\vee)$, Lemma \ref{l:untwisted}
implies $a'_i  = \ol d (b'_i) \oplus c_i(1\otimes 1)$ with $b_i' \in S(V^\vee) \otimes C(V^\vee)$
and $c_i \in \bC$.  Using Lemmas \ref{l:oddder}
and \ref{l:winc} once again, 
we have
\begin{align*}
a_i &= \ol\rho(\wti w_i)a'_i \\
&= \ol \rho(\wti w_i)\left ( \ol d'(b'_i) + c_i(1 \otimes 1)\right )\\
&=\ol d (\rho(\wti w_i)b'_i) + c_i\ol \rho(\wti w_i) \\
&\in \Im(\ol d) + \rho(\bC[\wti W]).
\end{align*}
Hence $a = a_1 + \cdots + a_l \in \Im(\ol d) + \rho(\bC[\wti W])$ and the proof is complete.
\end{proof}

The considerations around \eqref{e:winv} also apply in the graded
setting.
In particular, using an argument as in the proof of Lemma \ref{l:winvdirac},
we conclude
$\ol d$ interchanges
the spaces
\begin{equation}
\label{e:winvgr}
(\ol \bH \otimes C(V^\vee))^{\triv} = \{a \in \ol \bH \otimes C(V^\vee) 
\; | \; \rho(\wti w) a = a \rho(\wti w) \}
\end{equation}
and
\begin{equation}
\label{e:winvsgngr}
(\ol \bH \otimes C(V^\vee))^{\sgn} = \{a \in \ol \bH \otimes C(V^\vee) 
\; | \; \rho(\wti w) a = \sgn(\wti w) a \rho(\wti w) \}.
\end{equation}
As before, let $\ol d^{\triv}$ (resp.~$\ol d^{\sgn}$) 
denote the restriction of $d$ to the space
in \eqref{e:winv} (resp.~\eqref{e:winvsgn}).
Passing to the subspace
\[
(\ol \bH \otimes C(V^\vee))^{\triv} \oplus
\ol \bH \otimes C(V^\vee))^{\sgn}
\]
in Proposition \ref{p:dbar} we obtain the following corollary.

\begin{corollary}
\label{c:dbar}
With notation as in the previous paragraph,
\[
\ker(\ol d^{\triv}) = \Im(\ol d^{\sgn}) \oplus 
\ol \rho(\bC[\wti W]^{\wti W}).
\]
\end{corollary}

\medskip

\noindent Theorem \ref{t:hp2}, and hence
Theorem \ref{t:hp}, now follow from Corollary \ref{c:dbar} by an easy induction
based on the degree of the filtration.\qed

\begin{remark} 
\label{r:geom}

As remarked above, our proof shows that Theorem \ref{t:hp} 
holds for graded affine Hecke algebras attached 
to arbitrary
root systems and arbitrary
parameters.
The proof of Theorem \ref{t:vogan} depends on two
other key ingredients:
Theorem \ref{t:class} and the classification (and $W$-structure)
of tempered modules.
Both results are available for the algebras considered
by Lusztig in \cite{lu:1}, the former by 
\cite[Theorem
  3.10.1]{ciubo:weyl} and the latter by \cite{lu:3}.
Thus our proof establishes Theorem \ref{t:vogan} for cases of the unequal
parameters as in \cite{lu:1}.
\end{remark}

\ifx\undefined\bysame
\newcommand{\bysame}{\leavevmode\hbox to3em{\hrulefill}\,}
\fi


\begin{thebibliography}{EFM}


\bibitem[BM1]{BM1}
D.~Barbasch, A.~Moy,
\emph{A unitarity criterion for $p$-adic groups}, Invent. Math.
98 (1989), no. 1, 19-37.

\bibitem[BM2]{BM}
\bysame,
\emph{Reduction to real infinitesimal character in affine Hecke
  algebras}, J. Amer. Math. Soc. 6 (1993), no.~3, 611--635. 



\bibitem[C]{ciubo:weyl}
D.~Ciubotaru, {\em Spin representations of Weyl groups and the Springer correspondence}, preprint.

\bibitem[CT]{ct2}
D.~Ciubotaru, P.~E.~Trapa, {\em Functors for unitary representations of
classical real and $p$-adic groups}, {\tt arXiv:0906.2378}, to appear in Adv.~Math.




\bibitem[EFM]{EFM} P.~Etingof, R.~Freund, X.~Ma, 
{\em  
A Lie-theoretic construction of some representations of the degenerate affine 
and double affine Hecke algebras of type $BC_n$},
Represent.~Theory 13 (2009), 33--49. 


\bibitem[HP]{HP}
J.-S.~Huang, P.~Pand\v zi\'c,
\emph{Dirac cohomology, unitary representations and a proof of a conjecture of Vogan},
J. Amer. Math. Soc. 15 (2001), no.~1, 185--202.

\bibitem[K]{ko}
B.~Kostant, {\em Dirac cohomology for the cubic Dirac operator}, in
{\em Studies in memory of Issai Schur}, Progress in Math. vol.~210 (2002),
Birkha\"{u}ser (Boston).

\bibitem[KL]{KL}
D.~Kazhdan, G.~Lusztig, \emph{Proof of the Deligne-Langlands
  conjecture for affine Hecke algebras}, Invent.~Math.
87 (1987), no.~1, 153--215.


\bibitem[Lu1]{L}
G.~Lusztig,
\emph{Affine Hecke algebras and their graded version},
J. Amer. Math. Soc. 2 (1989), 599--635. 

\bibitem[Lu2]{lu:1}
G.~Lusztig,
{\em Cuspidal local systems and graded Hecke algebras. I},
Inst. Hautes Études Sci. Publ. Math. 67 (1988), 145--202. 
 
\bibitem[Lu3]{lu:3} \bysame \emph{Cuspidal local systems and graded
  Hecke algebras. III}, Represent. Theory 6 (2002), 202--242.

\bibitem[Mo1]{Mo} A.~Morris, \emph{Projective representations of
  reflection groups. II}, Proc. London Math. Soc. 40 (1980), no. 3,
  553--576.


\bibitem[Pa]{Pa} R.~Parthasarathy, \emph{Dirac operator and the
  discrete series}, Ann. of Math. (2) 96 (1972), 1--30.



\bibitem[SV]{SV}
S. Salamanca-Riba, D.A. Vogan, Jr., 
\emph{On the classification of unitary representations of reductive Lie groups}, Ann. of Math. (2) 148 (1998), no. 3, 1067--1133.


\bibitem[V]{v:notes}
D.A. Vogan, Jr., {\em Dirac operators and unitary representations I--III}, lectures at MIT, 1997.




\end{thebibliography}
\end{document}